\documentclass[11pt,reqno]{amsart}
\usepackage[margin=1in]{geometry}
\usepackage{amsmath,amssymb,amsthm,mathtools,comment}
\usepackage{enumitem}
\usepackage{microtype}
\usepackage[colorlinks=true,linkcolor=blue,citecolor=blue,urlcolor=blue]{hyperref}

\newtheorem{theorem}{Theorem}[section]
\newtheorem{proposition}[theorem]{Proposition}
\newtheorem{lemma}[theorem]{Lemma}
\newtheorem{remark}[theorem]{Remark}
\newtheorem{Op}{Open Problem}[section]

\newcommand{\R}{\mathbb R}
\newcommand{\T}{\mathbb T}
\newcommand{\Z}{\mathbb Z}
\newcommand{\N}{\mathbb N}
\newcommand{\BUC}{\mathrm{BUC}}
\newcommand{\Lip}{\mathrm{Lip}}
\newcommand{\cH}{\overline H}
\newcommand{\eps}{\varepsilon}
\newcommand{\dist}{\operatorname{dist}}

\newcommand{\Hbar}{\overline H}
\newcommand{\abs}[1]{\left|#1\right|}
\newcommand{\norm}[1]{\left\lVert #1\right\rVert}
\newcommand{\la}{\lambda}

\title[Optimal convergence rates]{Optimal convergence rates in periodic homogenization of nonconvex Hamilton--Jacobi equations}
\author[J. Jang, Q. Sun, H. V. Tran, Y. Yu]{Jiwoong Jang, Qi Sun, Hung V. Tran, Yifeng Yu}

\date{}

\thanks{
H. V. Tran is partially supported by NSF grant DMS-2348305. 
}

\address[J. Jang]
{
Department of Mathematics, University of Maryland-College Park, William E. Kirwan Hall, 4176 Campus Drive, College Park, Maryland 20742, USA}
\email{jjang124@umd.edu}

\address[Q. Sun]
{Department of Mathematics, Rutgers University, 110 Frelinghuysen Road Piscataway, New Jersey 08854, USA}
\email{qs176@math.rutgers.edu}

\address[H. V. Tran]
{
Department of Mathematics, 
University of Wisconsin-Madison, Van Vleck Hall, 480 Lincoln Drive, Madison, Wisconsin 53706, USA}
\email{hung@math.wisc.edu}

\address[Y. Yu]
{
Department of Mathematics, 
University of California at Irvine, 
California 92697, USA}
\email{yifengy@uci.edu}



\begin{document}

\begin{abstract}
We study the convergence rates in periodic homogenization of general nonconvex, coercive Hamilton--Jacobi equations.
We show that the optimal convergence rate is $O(\eps^{1/2})$ in one dimension, $O(\eps^{1/3})$ in two dimensions (up to a logarithmic factor), and $O(\eps^{1/3})$ in dimension three or higher.
\end{abstract}

\maketitle

\section{Introduction}

Let $H \in C(\R^n\times\R^n,\R)$ be $\Z^n$-periodic in its first variable and coercive in its second, and consider, for $\eps\in (0,1)$,
\begin{equation}\label{eq:micro}
\begin{cases}
  u_t^\eps+H\left(\frac{z}{\eps},Du^\eps\right)=0 \qquad &\text{in }\R^n\times(0,\infty),\\
  u^\eps(z,0)=g(z)\qquad &\text{on }\R^n.
\end{cases}
\end{equation}
Here, $g\in\BUC(\R^n)\cap\Lip(\R^n)$ is the given initial datum.
By the standard periodic homogenization theory of Hamilton--Jacobi equations \cite{LPV, Evans1992}, $u^\eps$ converges locally uniformly to the solution of
\begin{equation}\label{eq:eff}
\begin{cases}
  u_t+\Hbar(Du)=0  \qquad &\text{in }\R^n\times(0,\infty),\\
  u(z,0)=g(z)\qquad &\text{on }\R^n.
\end{cases}
\end{equation}
The effective Hamiltonian $\Hbar$ is determined by the periodic cell problem; see, for example, \cite{LPV, Evans1992,Tran}. 
For general nonconvex, globally Lipschitz Hamiltonians, the classical quantitative theory \cite{CDI} gives an $O(\eps^{1/3})$ convergence rate (see also \cite{Tran}). 

\subsection{Main results}

We first show that in three dimensions and higher, the optimal convergence rate is precisely $O(\eps^{1/3})$.
It is enough to obtain the result in three dimensions.

\begin{theorem}\label{main-thm}
Assume $n=3$.
There exist a Hamiltonian $H\in C^{0,1}(\R^3\times\R^3)$, which is periodic in its first variable and coercive and nonconvex in its second, and an initial datum $g\in\BUC(\R^3)\cap\Lip(\R^3)$ with the following
property: there are constants $c,C,\eps_0>0$ such that the solutions of \eqref{eq:micro}--\eqref{eq:eff} satisfy
 \[
  c\eps^{1/3}
  \leq
  \|u^\eps-u\|_{L^\infty(\R^3\times[0,1])}
  \leq C\eps^{1/3}
  \qquad(0<\eps\leq\eps_0).
 \]
\end{theorem}

In two dimensions, we also have that the power $1/3$ is optimal, up to a logarithmic factor.

\begin{theorem}\label{thm:2D-optimal}
Assume $n=2$.  There exist a Hamiltonian
$H\in C^{0,1}(\R^2\times\R^2)$, which is $\Z^2$-periodic in its first
variable and coercive and nonconvex in its second, and an initial datum
$g\in\BUC(\R^2)\cap\Lip(\R^2)$ such that, for some $c,\eps_0>0$,
\[
 \|u^\eps-u\|_{L^\infty(\R^2\times[0,1])}
 \geq
 c\,\frac{\eps^{1/3}}
 {\log^{4/3}(e+\eps^{-1})}
 \qquad (0<\eps\leq\eps_0).
\]
In particular, for every $\theta>1/3$,
\[
 \lim_{\eps\to0}
 \frac{\|u^\eps-u\|_{L^\infty(\R^2\times[0,1])}}
 {\eps^\theta}=+\infty.
\]
\end{theorem}

Concerning the one-dimensional case, we show that there is a Lipschitz selection of correctors, which implies the $O(\eps^{1/2})$ convergence rate.

\begin{theorem}\label{thm:1D-Lip}
Assume $n=1$.
Let $H \in C(\R\times\R,\R)$ be $\Z$-periodic in its first variable and coercive in its second.  
Then one can select, simultaneously for all $p\in\R$, periodic viscosity solutions $v(\cdot,p)$ of 
\begin{equation}\label{eq:cell}
 H\bigl(y,p+v_y(y,p)\bigr)=\Hbar(p)
 \qquad\text{in }\T,
\end{equation}
normalized by
\[
 v(0,p)=0,
\]
so that for every $p<q$,
\begin{equation}\label{eq:slope-order}
 p+v_y(\cdot,p)\le q+v_y(\cdot,q)
 \qquad\text{a.e. on }\T.
\end{equation}
In particular,
\begin{equation}\label{eq:Lip-selection}
 \norm{v(\cdot,p)-v(\cdot,q)}_{L^\infty(\T)}
 \le \abs{p-q}
 \qquad\text{for all }p,q\in\R.
\end{equation}

Assume further that $H \in C^{0,1}(\R\times\R,\R)$.
Then, for $\eps\in (0,1)$ and $T>0$, there exists a constant $C=C(H,\|Dg\|_{L^\infty},T)$ such that the solutions of \eqref{eq:micro}--\eqref{eq:eff} satisfy
 \[
  \|u^\eps-u\|_{L^\infty(\R\times[0,T])}
  \leq C\eps^{1/2}.
 \]
\end{theorem}

We have that this $O(\eps^{1/2})$ rate is optimal in one dimension; see Theorem \ref{thm:1D-optimal}.

\subsection{Relevant literature}
Qualitative periodic homogenization of Hamilton--Jacobi equations was proved in \cite{LPV, Evans1992}.
For general nonconvex, globally Lipschitz Hamiltonians, the classical quantitative theory \cite{CDI} gives an $O(\eps^{1/3})$ convergence rate.
If there exists a Lipschitz selection of correctors, then the convergence rate is upgraded to $O(\eps^{1/2})$ (see \cite[Theorem 4.40]{Tran}).
It has been a long-standing open problem to determine the optimal convergence rates in general nonconvex settings.

\smallskip

By now, the convex setting is well understood, with the optimal convergence rate $O(\eps)$ obtained in \cite{TY2022} by using the optimal control formula and the subadditivity and superadditivity of the metric functions.
The superadditivity was proved thanks to the equal curve cutting lemma in \cite{Burago}.
Earlier, \cite{MTY} proved an optimal lower bound in all dimensions, an optimal upper bound for positively homogeneous Hamiltonians in two dimensions, and a rather restrictive conditional upper bound in all dimensions by using
optimal control formulas, weak KAM theory, and Aubry--Mather theory;
and \cite{Coop1} obtained a near optimal rate $O(\eps|\log\eps|)$ through the optimal control framework in \cite{MTY} combined with Alexander’s theorem from first passage percolation.
Since then, the subadditivity and superadditivity of the metric functions' viewpoint has been adapted to multiscale and space-time periodic Hamiltonians \cite{HanJang, HNguyen}, to the state-constraint, Neumann, and Dirichlet problems on perforated domains \cite{HJMT, MiNi1, HanTu}, to Hamiltonians that are periodic in the unknown with applications to dislocation \cite{MNT}, and to some weakly coupled systems \cite{MiNi2}.
This viewpoint also works in the infinite-dimensional settings with the optimal $O(\eps)$ rate obtained under some appropriate conditions: see \cite{DEHZ} for the results in the Wasserstein space $\mathcal{P}_2(\R^d)$ and \cite{Park2} for the result in the Hilbert space $L^2([0,1]^d,\R^d)$.

For other recent quantitative homogenization results related to the optimal control and metric problem frameworks, we refer the reader to \cite{Tu, JTY, HTZ, GJTZ}.
For the nonconvex setting in the Hilbert space $L^2([0,1]^d,\R^d)$ with rearrangement invariance, an $O(\eps^{1/3})$ rate was obtained in \cite{Park1}.
For the second-order equations, an $O(\eps^{1/3})$ rate was obtained in \cite{CCM}; and in the uniformly elliptic case where there is a Lipschitz selection of correctors, an $O(\eps^{1/2})$ convergence rate was proved in \cite{QSTY,Jang}.
For viscous quadratic Hamilton-Jacobi equations, the sharp $O(\eps|\log \eps|)$ rate was obtained in \cite{LTY1}.

\smallskip

For general nonconvex settings, \cite{MTY} conjectures that the optimal rate is $O(\eps^{1/2})$ in one dimension.
Afterwards, \cite{Tran-conj} made a bolder conjecture that the optimal rate is $O(\eps^{1/2})$ in all dimensions.
\cite{Coop2} gave an example where the convergence rate is $\Omega(\eps^{1/2})$ in two dimensions.
The main results in this paper resolve the conjectures in \cite{MTY,Tran-conj} fully as follows.
\begin{itemize}
    \item In one dimension, Theorems \ref{thm:1D-Lip} and \ref{thm:1D-optimal} confirm the conjectures in \cite{MTY,Tran-conj}, and that the optimal convergence rate is $O(\eps^{1/2})$.
    \item In two dimensions and higher, Theorems \ref{main-thm} and \ref{thm:2D-optimal} disprove the conjecture in \cite{Tran-conj}.
    Rather surprisingly, Theorems \ref{main-thm} and \ref{thm:2D-optimal} confirm that the convergence exponent $1/3$ obtained in \cite{CDI} is optimal in dimension two or higher.
\end{itemize}

\subsection{Open problems}
It is rather intriguing that the optimal convergence rate is $O(\eps)$ in the convex case, and is only $O(\eps^{1/3})$ in the general nonconvex case in dimension three or higher.
It is unclear to us yet about the gap in terms of exponent between $1/3$ and $1$.
We list here several open problems that are of interest.

\begin{Op}\label{Op1}
Assume that $H\in C^{0,1}(\R^n\times\R^n)$, which is periodic in its first variable and coercive and level-set quasiconvex in its second.
What are the optimal convergence rates of the homogenization problem in this case?
\end{Op}

\begin{Op}\label{Op2}
Let $n=2$.  
Assume $H\in C^{0,1}(\R^2\times\R^2)$, which is $\Z^2$-periodic in its first
variable and coercive and nonconvex in its second, and $g\in\BUC(\R^2)\cap\Lip(\R^2)$.
Is it true that
\[
 \|u^\eps-u\|_{L^\infty(\R^2\times[0,1])}= o(\eps^{1/3})\,?
\]

\end{Op}

\subsection*{Organization of the paper}
We prove Theorem \ref{main-thm} in Section \ref{sec:3D}.
In Section \ref{sec:2D}, we consider the two-dimensional case and give the proof of Theorem \ref{thm:2D-optimal}.
The one-dimensional case and the proofs of Theorems \ref{thm:1D-Lip} and \ref{thm:1D-optimal} are given in Section \ref{sec:1D}.
In Appendix \ref{appendix}, we give an example in two dimensions to show that there is no continuous selection of correctors in two dimensions.
Because of this, it is not possible to use the approach of \cite{CDI} directly to obtain the convergence rate $o(\eps^{1/3})$ posed in Open Problem \ref{Op2}.

\subsection*{AI assistance}
During an early exploratory stage, we used OpenAI’s GPT-5.6 Sol, through a series of chats, to suggest possible choices of Hamiltonians. 
These exchanges provided some preliminary directions, but none of the first AI-generated constructions or arguments was used in the final manuscript. 
The final constructions and all mathematical arguments were developed, written, and verified by the authors. 
The authors take full responsibility for the correctness and content of the paper.

\section{The three-dimensional case}\label{sec:3D}
We give the proof of Theorem \ref{main-thm} in this section.
Note that the bound $\|u^\eps-u\|_{L^\infty(\R^3\times[0,1])}
  \leq C\eps^{1/3}$ was already obtained in \cite{CDI,Tran}.
\subsection{Initial constructions}
Fix $p_1 \in (0,10^{-3})$ and set
\[
 p_j=2^{1-j}p_1,\qquad
 a_j=\frac{1}{4}+p_j<\frac12,\qquad
 \rho_j=\frac{p_j}{100}
 \qquad(j\geq1).
\]
For $j\in \N$, fix an even integer $N_j$ such that
\[
 N_j \in \left[\frac1{20p_j},\frac1{10p_j}\right].
\]
As the length of each interval is at least $1/(20p_1) \geq 50$, the existence of $N_j$ is guaranteed.
On the circle of radius $a_j$, define two interlaced sets of discrete points in $\R^2$:
\begin{align}
 K_{0,j}
 &=
 \left\{
 a_j\left(\cos\frac{2\pi k}{N_j},
          \sin\frac{2\pi k}{N_j}\right):
 0\leq k<N_j
 \right\},                                               \label{K0}\\
 K_{1,j}
 &=
 \left\{
 a_j\left(\cos\frac{2\pi(k+1/2)}{N_j},
          \sin\frac{2\pi(k+1/2)}{N_j}\right):
 0\leq k<N_j
 \right\}.                                               \label{K1}
\end{align}
Because $N_j$ is even, both sets are centrally symmetric.


For $c,P\in\R^2$ and $\rho>0$, let
\[
 \tau_{c,\rho}(P)=(\rho-|P-c|)_+.
\]
Define, for $i=0,1$ and $P\in \R^2$,
\begin{equation}
 h_i(P)=
 \sum_{j=1}^{\infty}\ \sum_{c\in K_{i,j}}
 \tau_{c,\rho_j}(P),                                           \label{hi-def}
\end{equation}
which will play an important role in constructing the Hamiltonian $H$.
\begin{lemma}\label{packing}
The functions $h_0,h_1$ are nonnegative, bounded, even, and globally one-Lipschitz.  Moreover,
\begin{align}
 \min\{h_0(P),h_1(P)\}&=0 &&(P\in\R^2),                  \label{minzero}\\
 h_i(c)&=\rho_j &&(c\in K_{i,j}),                        \label{peak}\\
 h_i(P)&=0 &&(|P|\geq\tfrac12).                          \label{outerzero}
\end{align}
\end{lemma}

\begin{proof}
For a fixed $j$, the union $K_{0,j}\cup K_{1,j}$ is a regular
$2N_j$-gon.  
The distance between two consecutive points in this union is
\[
 2a_j\sin\frac{\pi}{2N_j}
 \geq \frac{2a_j}{N_j}
 \geq 20a_jp_j
 \geq 5p_j.
\]
Here, we used $\sin s\geq2s/\pi$ on $[0,\pi/2]$ and
$N_j\leq1/(10p_j)$.  
This distance is greater than $2\rho_j$.

If $k>j$, then
\[
 a_j-a_k=p_j-p_k\geq\frac{p_j}{2},
\]
whereas
\[
 \rho_j+\rho_k\leq\frac{3 p_j}{200}<\frac{p_j}{2}.
\]
Consequently, all closed balls
\[
 \overline B_{\rho_j}(c),
 \qquad c\in K_{0,j}\cup K_{1,j},\quad j\geq1,
\]
are pairwise disjoint; see Figure \ref{fig:closed-balls-discrete-rings}. They are contained in $\{|P|<a_1+\rho_1\}\subset\{|P|<1/2\}$.

\begin{figure}[htbp]
    \centering
    \includegraphics[width=\linewidth]
        {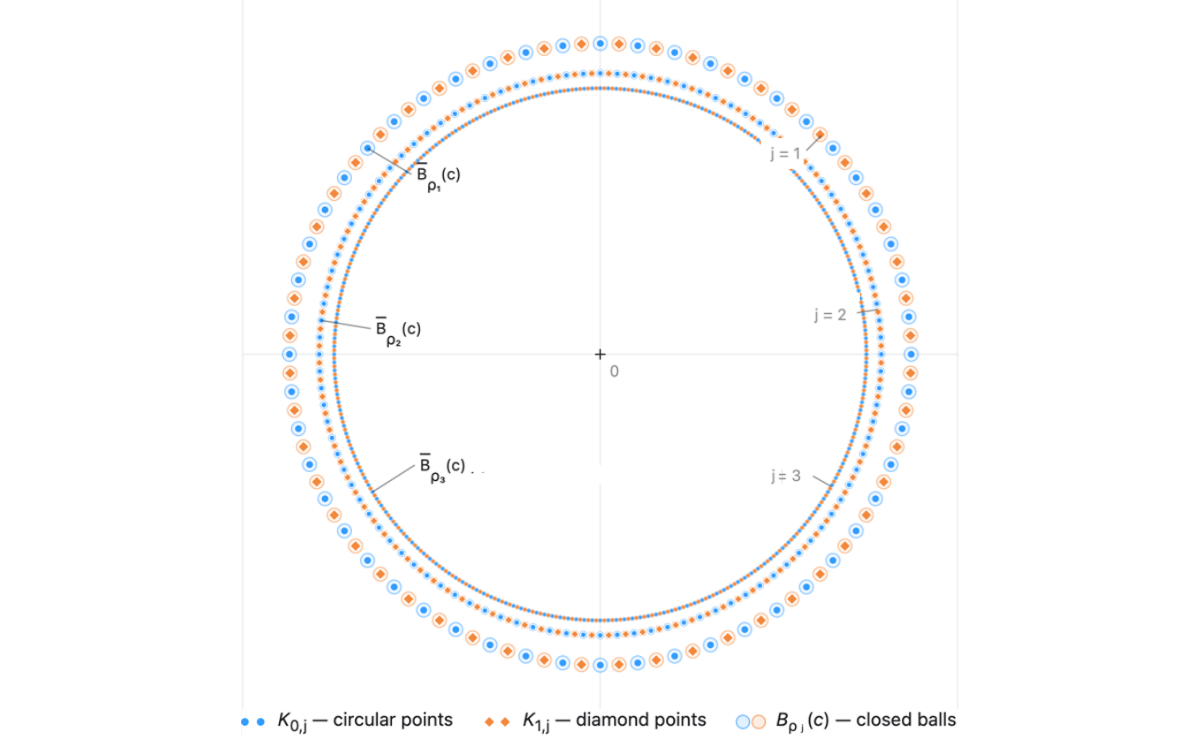}
    \caption{The closed balls
    $\overline{B}_{\rho_j}(c)$ centered at
    $c\in K_{0,j}\cup K_{1,j}$ for $j=1,2,3$.
    Radial offsets are enlarged by $60$ and ball radii by $700$
    for visibility.}
    \label{fig:closed-balls-discrete-rings}
\end{figure}

Thus at most one term in the two sums \eqref{hi-def} is positive at any point.  
This immediately proves \eqref{minzero}--\eqref{outerzero}.  Each cone bump $\tau_{c,\rho_j}$ is one-Lipschitz, and thus, $h_i$ is globally one-Lipschitz.
Finally, central symmetry of every $K_{i,j}$ proves that $h_i$ is even.
\end{proof}

For $i=0,1$, define the support function
\[
 \sigma_{i,j}(X)=\max_{P\in K_{i,j}}P\cdot X,
 \qquad X\in\R^2.
\]

\begin{lemma}
\label{support-close}
One has
\begin{align}
 a_j\cos\frac{\pi}{N_j}|X|
 \leq \sigma_{i,j}(X)
 \leq a_j|X|,                                            \label{support-bounds}\\
 |\sigma_{0,j}(X)-\sigma_{1,j}(X)|
 \leq100 \pi^2p_j^2|X|.                                \label{support-diff}
\end{align}
\end{lemma}

\begin{proof}
For any direction $X/|X|$, one of the vertices of either regular polygon has angular distance at most $\pi/N_j$ from that direction.  
This proves \eqref{support-bounds}.  
Both support functions $\sigma_{0,j},\sigma_{1,j}$ therefore lie in the interval $\left[
 a_j\cos\frac{\pi}{N_j}|X|,\ a_j|X|
 \right]$ and thus
\[
 |\sigma_{0,j}(X)-\sigma_{1,j}(X)|
 \leq a_j\left(1-\cos\frac{\pi}{N_j}\right)|X|\leq \frac{a_j\pi^2}{2N_j^2}|X|.
\]
Since $N_j\geq1/(20p_j)$ and $a_j=\frac{1}{4}+p_j<\frac12$,
\[
 |\sigma_{0,j}(X)-\sigma_{1,j}(X)|
 \leq100 \pi^2p_j^2|X|.
\]
\end{proof}

\subsection{Explicit subsolutions}

For $i=0,1$ and $j\geq1$, set
\begin{equation}
 w_{i,j}(X,t)
 =
 \max\left\{
 |X|,\ \rho_jt+\sigma_{i,j}(X)
 \right\}.
 \label{wij}
\end{equation}

\begin{lemma}\label{barriers}
For each $i,j$, the function $w_{i,j}$ is a viscosity subsolution of
\begin{equation}
 (w_{i,j})_t-h_i(D_Xw_{i,j})=0
 \qquad\hbox{in }\R^2 \times (0,\infty).                    \label{two-d-pde}
\end{equation}
It also satisfies
\begin{align}
 w_{i,j}(X,0)&=|X|,                                     \label{barrier-data}\\
 w_{i,j}(0,t)&=\rho_jt,                                  \label{barrier-center}\\
 \Lip_X w_{i,j}&\leq 1.                                  \label{barrier-lip}
\end{align}
\end{lemma}

\begin{proof}
Note that
\[
 \rho_jt+\sigma_{i,j}(X)
 =
 \max_{c\in K_{i,j}}\{\rho_jt+c\cdot X\}.
\]
For every $c\in K_{i,j}$, the affine function
$\rho_jt+c\cdot X$ solves \eqref{two-d-pde} classically because
$h_i(c)=\rho_j$.  
Besides, as $h_i \geq 0$, $\varphi(X,t)=|X|$ is a viscosity subsolution of \eqref{two-d-pde}.
Hence, 
\[
w_{i,j}(X,t) =
 \max\left\{
 |X|,\ \rho_jt+\sigma_{i,j}(X)
 \right\}= 
 \max\left\{
 \varphi(X,t),\ \rho_jt+\sigma_{i,j}(X)
 \right\}
\]
is a viscosity subsolution of \eqref{two-d-pde}.

Since $a_j<1/2$, \eqref{support-bounds} gives
$\sigma_{i,j}(X)\leq a_j|X|\leq |X|$, proving
\eqref{barrier-data}.  Formula \eqref{barrier-center} is immediate. Combined with \eqref{support-diff}, all affine slopes have norm at most $1$, which proves
\eqref{barrier-lip}.
\end{proof}

\begin{lemma}\label{compatibility}
For $(X,t)\in \R^2\times [0,\infty)$,
\begin{equation}
 |w_{0,j}(X,t)-w_{1,j}(X,t)|
 \leq 2 \pi^2p_j^3t.                             \label{compat-est}
\end{equation}
Fix $\mu>0$ and define
\begin{equation}
 T_j=\frac{\mu}{2 \pi^2p_j^3}.                               \label{Tj}
\end{equation}
Then, $T_{j+1}=8T_j$ for $j\in \N$, and
\begin{equation}
 |w_{0,j}-w_{1,j}|\leq\mu
 \qquad\hbox{on }\R^2\times [0,T_j].                               \label{compat-window}
\end{equation}
\end{lemma}

\begin{proof}
Write $A=|X|$ and $B_i=\rho_jt+\sigma_{i,j}(X)$. The elementary inequality
\[
 |\max\{A,B_0\}-\max\{A,B_1\}|
 \leq |B_0-B_1|
\]
and Lemma~\ref{support-close} give
\begin{equation}
 |w_{0,j}-w_{1,j}|
 \leq100 \pi^2p_j^2|X|.                                \label{first-compat}
\end{equation}
If both maxima equal $A$, their difference is zero.  Otherwise, for at least one $i$ we have $A\leq B_i$, which is to say $|X|\leq\rho_jt+\sigma_{i,j}(X)$. Since $\sigma_{i,j}(X)\leq a_j|X| \leq |X|/2$, one has
\[
 |X|\leq 2 \rho_j t=\frac{p_jt}{50}.
\]
Combining this with \eqref{first-compat} proves \eqref{compat-est}. 
The rest follows from \eqref{Tj} and $p_{j+1}=p_j/2$.
\end{proof}

\begin{figure}[htbp]
    \centering
    \includegraphics[width=1.00\textwidth]{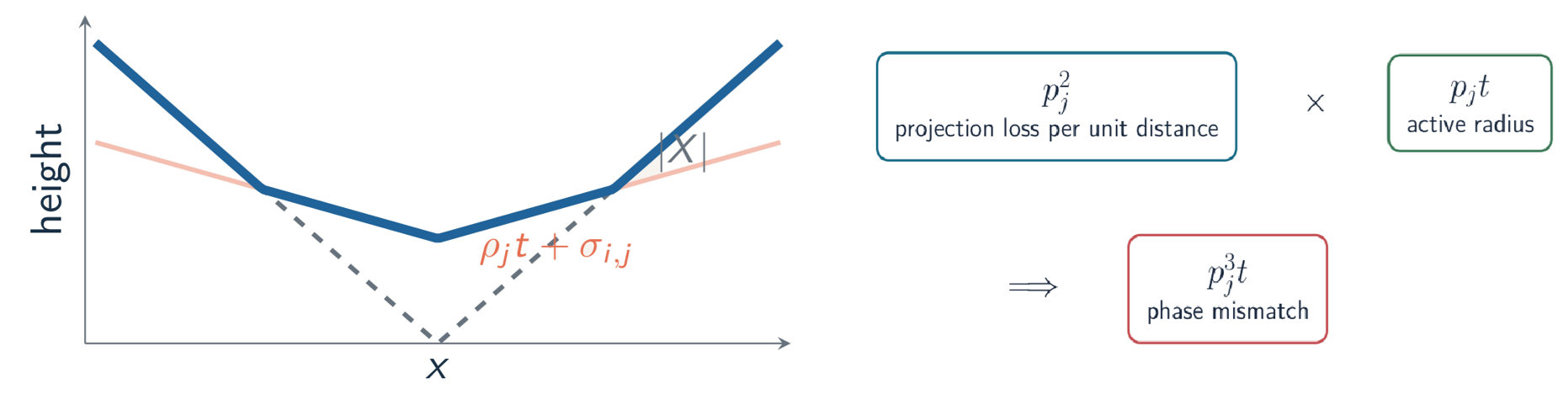}
    \caption{The graph of $w_{i,j}$ in blue and the power counting in Lemma \ref{compatibility}.}
    \label{fig:active_region}
\end{figure}

\subsection{Construction of $H$}
We now introduce the periodic variable $y\in\T$ to encode switching between two modes $h_0$ and $h_1$.

Choose $\gamma_0=0$, $ \gamma_1=1/2$, and write, for $i=0,1$,
\[
 d_i(y)=d_{\T}(y,\gamma_i),\qquad
 r(y)=\min\{d_0(y),d_1(y)\}.
\]
Then, for $y\in \T$,
\begin{equation}\label{distance-sum}
 d_0(y)+d_1(y)=\frac12.                              
\end{equation}
Let $\omega(y)=2d_0(y)$ and, for $y\in \T$ and $P\in \R^2$,
\[
h(y,P)=(1-\omega(y))h_0(P)+\omega(y)h_1(P).
\]
Thus, $h(\gamma_i,P)=h_i(P)$.  
It is clear that, for $y,y'\in\T$ and $P\in\R^2$,
\begin{align}
 &|h(y,P)-h(y',P)|
 \leq |\omega(y)-\omega(y')| (|h_0(P)|+|h_1(P)|) \notag\\
 \leq\,&  2d_{\T}(y,y') (2 \rho_1)\leq d_{\T}(y,y').                    \label{Ly}
\end{align}

Fix $\kappa=4\mu$.
For $\xi=(x_1,x_2,y)\in\R^3$ and momentum $(P,q)\in\R^2\times\R$, define the Hamiltonian
\begin{equation}
 H(\xi,(P,q))
 =
 (|q|-\kappa)_+
 -h(y,P)-4r(y)
 +(|P|-1)_+.
 \label{H-def}
\end{equation}
The first two components of $\xi$ do not occur in the formula, so $H$ is
$\Z^3$-periodic in $\xi$.
Define
\begin{equation}
 W_j(X,y,t)
 =
 \max_{i=0,1}
 \left\{
 w_{i,j}(X,t)-\kappa d_i(y)
 \right\}.
 \label{Wj}
\end{equation}

\begin{lemma}\label{envelope}
The function $W_j$ is a viscosity subsolution of
\begin{equation}
 (W_j)_t+H\big((X,y),(D_XW_j,(W_j)_y)\big)=0
 \qquad\hbox{in }\R^2\times\T \times (0,T_j).               
 \label{three-d-pde}
\end{equation}
Moreover,
\begin{align}
 W_j(X,y,0)&=|X|-\kappa r(y)\leq |X|,                  \label{W-data}\\
 W_j(X,\gamma_i,t)&=w_{i,j}(X,t)
 \quad(0\leq t\leq T_j,\ i=0,1).                        \label{W-trace}
\end{align}
\end{lemma}

\begin{proof}
Let $\phi\in C^1$ touch $W_j$ from above at
$(X_0,y_0,t_0)$, with $0<t_0<T_j$, and choose an active index $i$ in
\eqref{Wj}.  Keeping $y=y_0$ fixed, the function
\[
 (X,t)\longmapsto\phi(X,y_0,t)+\kappa d_i(y_0)
\]
touches $w_{i,j}$ from above.  Lemma~\ref{barriers} gives
\begin{equation}
 \phi_t(X_0,y_0,t_0)
 \leq h_i(D_X\phi(X_0,y_0,t_0)).                         \label{phi-t}
\end{equation}
The spatial Lipschitz bounds in \eqref{Wj} imply
\begin{equation}
 |D_X\phi|\leq 1,\qquad |\phi_y|\leq\kappa.             \label{test-gradients}
\end{equation}

We claim that, with $P=D_X\phi(X_0,y_0,t_0)$,
\begin{equation}
 h_i(P)\leq h(y_0,P)+4r(y_0).                          \label{key-envelope}
\end{equation}
If $d_i(y_0)=r(y_0)$, this follows from \eqref{Ly}
\[
 h_i(P)=h(\gamma_i,P)
 \leq h(y_0,P)+d_i(y_0)
 \leq h(y_0,P)+4r(y_0).
\]

Suppose instead that the other index $k=1-i$ is nearer to $y_0$.  
Since $i$ is active,
\[
 \kappa(d_i(y_0)-d_k(y_0))
 \leq w_{i,j}(X_0,t_0)-w_{k,j}(X_0,t_0)
 \leq\mu
\]
by Lemma~\ref{compatibility}.  
From \eqref{distance-sum} and $\kappa=4\mu$, we obtain
\[
 r(y_0)=d_k(y_0)\geq\frac18.
\]
Using \eqref{Ly} and $d_i\leq1/2$,
\[
 h_i(P)-h(y_0,P)
 \leq \frac{1}{2}
 \leq 4 r(y_0).
\]
This proves \eqref{key-envelope}.

Both positive-part terms in \eqref{H-def} vanish by
\eqref{test-gradients}.  
Combining \eqref{phi-t} and \eqref{key-envelope} yields
\[
 \phi_t+H((X_0,y_0),(D_X\phi,\phi_y))\leq0,
\]
which is the required viscosity subsolution inequality.

At $t=0$, both $w_{i,j}$ equal $|X|$, so \eqref{W-data} is immediate.
At $y=\gamma_i$, the competing branch satisfies
\[
 w_{1-i,j}(X,t)-\frac{\kappa}{2}
 \leq w_{i,j}(X,t)+\mu-2\mu
 <w_{i,j}(X,t).
\]
Therefore the $i$th branch is active and \eqref{W-trace} follows.
\end{proof}

Let $U$ denote the unique viscosity solution, in the class of functions with at most linear growth in $X$, of
\begin{equation}\label{Ueq}
\begin{cases}
 U_t+H\big((X,y),(D_XU,U_y)\big)=0 \qquad &\text{in } \R^2\times\T\times (0,\infty),\\
 U(X,y,0)=|X|  \qquad &\text{on } \R^2\times\T.    \end{cases}
\end{equation}

\begin{proposition}\label{unscaled-growth}
There are constants $c_0,t_0>0$ such that
\begin{equation}
 U(0,\gamma_i,t)\geq c_0t^{2/3}
 \qquad(t\geq t_0,\ i=0,1).                             \label{growth}
\end{equation}
\end{proposition}

\begin{proof}
By Lemma~\ref{envelope} and the comparison principle,
\[
 U(X,y,t)\geq W_j(X,y,t)
 \qquad(0\leq t\leq T_j).
\]
Using \eqref{W-trace} and \eqref{barrier-center},
\begin{equation}
 U(0,\gamma_i,t)\geq\rho_jt
 \qquad(0\leq t\leq T_j).                               \label{scale-growth}
\end{equation}
We need to estimate $\rho_j=p_j/100$.
For $t>T_1$, choose the least $j\geq2$ such that $t\leq T_j$. Then, combined with Lemma \ref{compatibility},
\[
 \frac{T_j}{8}=T_{j-1}<t\leq T_j.
\]
By \eqref{Tj},
\begin{equation}
T_{j-1}=\frac{\mu}{2 \pi^2 p_{j-1}^3}=\frac{\mu}{2 \pi^2\left(2 p_j\right)^3}=\frac{\mu}{16 \pi^2 p_j^3} .
\end{equation}
Thus,
\[
 p_j>\left(\frac{\mu}{16 \pi^2 t}\right)^{1 / 3}.
\]
Consequently, \eqref{scale-growth} gives
\[
U(0,\gamma_i,t)
\geq \frac{p_jt}{100}
>
\frac{1}{100}
\left(\frac{\mu}{16\pi^2}\right)^{1/3}t^{2/3}
=
\frac{1}{200}
\left(\frac{\mu}{2\pi^2}\right)^{1/3}t^{2/3}.
\]
\end{proof}

\subsection{The cell problems and $\Hbar$}
 Fix $P\in\R^2$ with $|P|\leq 1$ and define
\[
 F_P(y)=h(y,P)+4r(y).
\]
By Lemma~\ref{packing}, at least one of $h_0(P),h_1(P)$ is zero.  
Hence,
\[
 Z_P=\{\gamma_i:h_i(P)=0\}
 = 
 \begin{cases}\left\{\gamma_0\right\}, \qquad& h_0(P)=0, h_1(P)>0, \\ \left\{\gamma_1\right\}, \qquad& h_1(P)=0, h_0(P)>0, \\ 
 \left\{\gamma_0, \gamma_1\right\}, \qquad& h_0(P)=h_1(P)=0 .\end{cases}
\]
is nonempty.  
Moreover, we have that $F_P\geq0$ and
\begin{equation}
F_P(y)=0\quad\Longleftrightarrow\quad y\in Z_P.          \label{Fzero}
\end{equation}
Fix $z\in Z_P$ and recall $\kappa=4\mu$. 
Choose \(z\in Z_P\), and fix a lift of \(z\) to \(\R\), still denoted by
\(z\). 
We regard \(F_P\) as a \(1\)-periodic function on \(\R\) and set
\[
g_P(y):=\kappa+F_P(y).
\]
Since \(g_P\geq\kappa>0\), the function
\[
s\longmapsto \int_z^s g_P(y)\,dy
\]
is continuous and strictly increasing on \([z,z+1]\). Hence there exists
a unique \(z^*\in(z,z+1)\) such that
\begin{equation}
    \int_z^{z^*} g_P(y)\,dy
    =
    \int_{z^*}^{z+1} g_P(y)\,dy
    =
    \frac12\int_z^{z+1}g_P(y)\,dy.
    \label{balanced-point}
\end{equation}

Define \(\chi_P\) on \([z,z+1]\) by
\[
\chi_P(y):=
\begin{cases}
\displaystyle
\int_z^y g_P(s)\,ds,
\qquad& z\leq y\leq z^*,\\[3mm]
\displaystyle
\int_z^{z^*}g_P(s)\,ds
-
\int_{z^*}^y g_P(s)\,ds,
\qquad& z^*\leq y\leq z+1.
\end{cases}
\]
By \eqref{balanced-point},
\[
\chi_P(z+1)=\chi_P(z)=0.
\]
Therefore, \(\chi_P\) extends to a \(1\)-periodic Lipschitz function on
\(\R\), which we continue to denote by \(\chi_P\). Its derivative
satisfies, almost everywhere,
\[
\chi_P'(y)=
\begin{cases}
\kappa+F_P(y),
\qquad& y\in(z,z^*),\\[1mm]
-\bigl(\kappa+F_P(y)\bigr),
\qquad& y\in(z^*,z+1).
\end{cases}
\]
\begin{lemma}\label{corrector-lemma}
The periodic Lipschitz function $\chi_P$ is a viscosity solution of
\begin{equation}
 (|\chi_P'|-\kappa)_+-F_P(y)=0
 \qquad\hbox{on }\T.                                     \label{eikonal}
\end{equation}
\end{lemma}

\begin{proof}
We note that the Hamiltonian $K(y,q)=(|q|-\kappa)_+ - F_P(y)$ in \eqref{eikonal} is convex in $q$.
Thanks to \cite[Chapter 2]{Tran}, we only need to check the supersolution property at $y=z$.
For $l\in D^-\chi_P(z)=[-\kappa, \kappa]$, it is clear that $|l|\le \kappa$.
Therefore,
\[
K(z,l)=(|l|-\kappa)_+ - F_P(z)= 0-0=0.
\]
Thus, $\chi_P$ is a viscosity solution of \eqref{eikonal}.
\end{proof}

\begin{proposition}\label{effective-flat}
For every $P\in\R^2$ with $|P|\leq 1$,
\begin{equation}
 \cH(P,0)=0.                                             \label{flat}
\end{equation}
\end{proposition}
\begin{proof}
Write \(\xi=(X,y)\in\T^2\times\T=\T^3\), and define
\[
v(X,y):=\chi_P(y).
\]
Since \(\chi_P\) is periodic and Lipschitz on \(\T\), the function \(v\)
is periodic and Lipschitz on \(\T^3\). Moreover, \(v\) is independent of
\(X\), so
\[
D_Xv=0,
\qquad
D_yv=D_y\chi_P.
\]
Consequently,
\[
(P,0)+Dv(\xi)
=
\bigl(P,D_y\chi_P(y)\bigr).
\]

We now evaluate the Hamiltonian defined in \eqref{H-def}. Since
\(|P|\leq1\),
\[
(|P|-1)_+=0.
\]
Recalling that
\[
F_P(y)=h(y,P)+4r(y),
\]
we obtain
\begin{align*}
H\bigl(\xi,(P,0)+Dv(\xi)\bigr)
&=
H\bigl(\xi,(P,D_y\chi_P(y))\bigr)\\
&=
\bigl(|D_y\chi_P(y)|-\kappa\bigr)_+
-h(y,P)-4r(y)\\
&=
\bigl(|D_y\chi_P(y)|-\kappa\bigr)_+
-F_P(y).
\end{align*}
By Lemma~\ref{corrector-lemma}, the last expression is zero in the
viscosity sense on \(\T\).
\end{proof}

\subsection{Proof of Theorem \ref{main-thm}}

The rescaling
\begin{equation}\label{scaling}
 u_{\rm cusp}^\eps(X,y,t)
 =
 \eps U\left(\frac{X}{\eps},\frac{y}{\eps},\frac{t}{\eps}\right)                    
\end{equation}
solves \eqref{eq:micro} with the initial datum $|X|$.  Proposition~\ref{unscaled-growth} yields, for all
sufficiently small $\eps$,
\begin{equation}
 u_{\rm cusp}^\eps(0,0,1)
 =
 \eps U(0,\gamma_0,\eps^{-1})
 \geq c_0\eps^{1/3}.                                    \label{cusp-lower}
\end{equation}
We now replace the cusp by one fixed bounded function without changing the value at $(0,0,1)$.

\begin{lemma}\label{finite-prop}
Let $G(z,p)$ be continuous, uniformly continuous in $z$ uniformly for bounded $p$, and globally $C_G$-Lipschitz in $p$.  
If $v_1,v_2$ are bounded viscosity solutions of
\[
\begin{cases}
    (v_k)_t+G(z,Dv_k)=0,\\
    v_k(\cdot,0)=g_k,
\end{cases}
\]
then
\begin{equation}
 v_1(z,t)-v_2(z,t)
 \leq
 \sup_{|z'-z|\leq C_Gt}\big(g_1(z')-g_2(z')\big).
 \label{finite-prop-est}
\end{equation}
In particular, if the initial data agree on $B_R(z)$ and $C_Gt<R$, then
the two solutions agree at $(z,t)$.
\end{lemma}
The proof of this lemma is standard and hence is skipped.
See \cite[Theorem 1.39]{Tran} for example.


\begin{proof}[Proof of Theorem \ref{main-thm}]
Both the Hamiltonian $H$ and the rescaled Hamiltonian $H(z/\eps,p)$ are Lipschitz in $(P,q)$ with Lipschitz constant at most $\sqrt{5}$.  
Choose $g\in\BUC(\R^3)\cap\Lip(\R^3)$ such that
\begin{equation}
 g(X,y)=\min\{|X|,4\}.                                \label{bounded-g}
\end{equation}

For $S>4$, let
\[
 g^S(X,y)=\min\{|X|,S\},
\]
and denote by $u^{\eps,S}$ the microscopic solution with initial datum $g^S$.
The functions $g^S$ and $g$ agree on $B_{4}(0)$. 
The finite-propagation lemma (Lemma~\ref{finite-prop}) therefore gives
\begin{equation}
 u^{\eps,S}(0,0,1)=u^\eps(0,0,1).                        \label{local-equality}
\end{equation}
Let $S\to \infty$, by the stability of viscosity solutions and the wellposedness of \eqref{eq:micro}  in the class of functions with at most linear growth in $X$,
\begin{equation}
 u^\eps(0,0,1)=\lim_{S\to \infty} u^{\eps,S}(0,0,1)=u_{\rm cusp}^\eps(0,0,1).                 \label{cusp-equality}
\end{equation}

It remains to identify the effective solution for the bounded datum \eqref{bounded-g}.  
The stationary function
\[
 u(X,y,t)=g(X,y)=\min\{|X|,4\}
\]
is the viscosity solution of \eqref{eq:eff}.  
Indeed, every test gradient has the form $(P,0)$ with $|P|\leq 1$, and Proposition~\ref{effective-flat} gives $\cH(P,0)=0$.  
In particular,
\begin{equation}
 u(0,0,1)=0.                                             \label{u-zero}
\end{equation}
Combining \eqref{cusp-lower}, \eqref{cusp-equality}, and \eqref{u-zero}
proves
\[
 \|u^\eps-u\|_{L^\infty(\R^3\times[0,1])}
 \geq u^\eps(0,0,1)-u(0,0,1)
 \geq c_0\eps^{1/3}.
\]
This completes the proof.
\end{proof}

\section{The two-dimensional case}\label{sec:2D}

We next show that the power $1/3$ is also optimal in two dimensions,
up to a logarithmic factor.

\begin{proof}[Proof of Theorem \ref{thm:2D-optimal}]
We only modify the momentum construction in the proof of
Theorem~\ref{main-thm}.  Put
\[
 \delta_j=\frac{\delta_*}{2^{j}},
 \qquad
 \ell_j=\frac{\ell_*}{(j+1)^2},
 \qquad j\geq1.
\]
Choose $\ell_*>0$ small and then $\delta_*>0$ so small that
$\delta_j\leq\ell_j/100$ for all $j$. In this proof, we take
\[
a_j := s+ \ell_j+ 2\sum_{k>j}\ell_k ,\qquad L:=2\sum_{k\geq1}\ell_k\in\left(0,\frac12\right),\qquad\text{and}\qquad s:=\frac14-\frac12L>0
\]
so that the following 
\[
 J_j=(a_j-\ell_j,a_j+\ell_j)\subset\left(0,\frac12\right)
\]
are pairwise disjoint intervals. Set
\[
 K_{0,j}=J_j\cap\{a_j-\ell_j+(8m+2)\delta_j:m\in\Z\},
 \qquad
 K_{1,j}=J_j\cap\{a_j-\ell_j+(8m+6)\delta_j:m\in\Z\}.
\]
Thus each $K_{i,j}$ is an $8\delta_j$-net of $J_j$, whereas two
points in $K_{0,j}\cup K_{1,j}$ are at distance at least
$4\delta_j$.  For $p\in\overline J_j$, let
\[
 R_j(p)=\delta_j\left(1-\frac{(p-a_j)^2}{\ell_j^2}\right),
\]
and define
\begin{equation}\label{eq:2D-hi}
 h_i(p)=\sum_{j\geq1}\sum_{c\in K_{i,j}}
       (R_j(c)-|p-c|)_+,
 \qquad i=0,1.
\end{equation}
See Figure \ref{fig:tent}. Since
\[
 R_j(c)\leq
 \frac{2\delta_j}{\ell_j}\dist(c,\partial J_j),
\]
every tent in \eqref{eq:2D-hi} is supported in $J_j$.  All the tent
supports, including those of different colors, are pairwise disjoint.
Since $\delta_j\to0$, the functions are also continuous at every
accumulation point of the intervals.  Consequently, $h_0,h_1$ are
nonnegative and one-Lipschitz, and
\begin{equation}\label{eq:2D-h-properties}
 \min\{h_0(p),h_1(p)\}=0,
 \qquad h_i(c)=R_j(c)\quad(c\in K_{i,j}),
 \qquad h_i(\pm1)=0.
\end{equation}
In particular, $h_0+h_1\leq\delta_1$.  Thus, by taking $\delta_*$
smaller, the interpolation in the three-dimensional construction is
one-Lipschitz in $y$, uniformly in $p$.

\begin{figure}[htbp]
    \centering
    \includegraphics[width=0.55\textwidth]{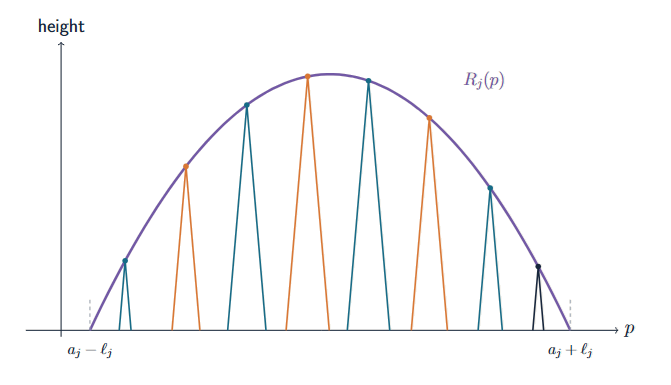}
    \caption{The union of the blue tents is the graph of $h_0$, while that of the red tents is the graph of $h_1$ (on a fixed $J_j$).}
    \label{fig:tent}
\end{figure}

For $i=0,1$, set
\begin{equation}\label{eq:2D-wij}
 w_{i,j}(x,t)=
 \max\left\{|x|,\ \max_{c\in K_{i,j}}
                 \{R_j(c)t+cx\}\right\}.
\end{equation}
Each affine function in this maximum solves
$w_t-h_i(w_x)=0$, and $|x|$ is a subsolution of the same equation.
It follows that $w_{i,j}$ is a viscosity subsolution.  Moreover,
\begin{equation}\label{eq:2D-barrier-basic}
 w_{i,j}(x,0)=|x|,\qquad
 \Lip_x w_{i,j}\leq1,\qquad
 w_{i,j}(0,t)\geq\frac12\delta_jt.
\end{equation}

We claim that
\begin{equation}\label{eq:2D-compatibility}
 |w_{0,j}(x,t)-w_{1,j}(x,t)|
 \leq64\frac{\delta_j^3}{\ell_j^2}t.
\end{equation}
Indeed, let
\[
 \Phi_j(x,t)=
 \max\left\{|x|,\ \sup_{p\in\overline J_j}
                 \{R_j(p)t+px\}\right\}.
\]
If the second term is active, its maximizer $p_*$ is an interior
point of $J_j$, since $R_j=0$ on $\partial J_j$ and $|p|<1$ there.
For each $i$, choose $c_i\in K_{i,j}$ with
$|c_i-p_*|\leq8\delta_j$.  The quadratic form of $R_j$ gives
\[
 0\leq \Phi_j(x,t)-w_{i,j}(x,t)
 \leq \frac{\delta_jt}{\ell_j^2}|c_i-p_*|^2
 \leq64\frac{\delta_j^3}{\ell_j^2}t.
\]
If the first term is active, both $w_{i,j}$ equal $|x|$.  This proves
\eqref{eq:2D-compatibility}.

Fix $\mu>0$ and put
\[
 T_j=\frac{\mu\ell_j^2}{64\delta_j^3}=\frac{\mu \ell_*^2}{\delta_*^3} \,\frac{2^{3(j-2)}}{(j+1)^4}.
\]
Then $|w_{0,j}-w_{1,j}|\leq\mu$ on $[0,T_j]$, and
\[
 \frac{T_{j+1}}{T_j}
 =8\left(\frac{j+1}{j+2}\right)^4
 \geq\frac{128}{81}.
\]

We now repeat \eqref{distance-sum}--\eqref{Wj}, with $X,P\in\R^2$
replaced by $x,p\in\R$ and with the functions in
\eqref{eq:2D-hi}--\eqref{eq:2D-wij}.  The proof of
Lemma~\ref{envelope} is unchanged and gives, for the solution with
initial datum $|x|$,
\begin{equation}\label{eq:2D-unscaled-growth}
 U(0,\gamma_i,t)\geq\frac12\delta_jt
 \qquad(0\leq t\leq T_j,\ i=0,1).
\end{equation}
The proof of Proposition~\ref{effective-flat} is also unchanged;
indeed, its only input from the momentum construction is
$\min\{h_0,h_1\}=0$.  Hence
\begin{equation}\label{eq:2D-effective-flat}
 \Hbar(p,0)=0\qquad(|p|\leq1).
\end{equation}
The resulting Hamiltonian is periodic, coercive, and globally
Lipschitz.  It is nonconvex.  Indeed, if $c\in K_{i,j}$ and
$R=R_j(c)$, then at $y=\gamma_i$ and $q=0$ its values at
$c,c+R,c+2R$ are $-R,0,0$, respectively, which violates the convex
midpoint inequality.

For $t>T_1$, choose the least $j\geq2$ such that $t\leq T_j$.
Then $t>T_{j-1}$, and therefore
\[
 \delta_j^3\geq c\frac{\ell_j^2}{t}.
\]
Since $T_j\asymp 8^j/(j+1)^4$, the minimality of $j$ also gives
$j\leq C\log(e+t)$.  Using \eqref{eq:2D-unscaled-growth}, we obtain
\[
 U(0,\gamma_i,t)
 \geq c\,t^{2/3}\ell_j^{2/3}
 \geq c\,\frac{t^{2/3}}{\log^{4/3}(e+t)}.
\]
Finally, the scaling, finite-propagation, and bounded-data arguments
in the proof of Theorem~\ref{main-thm}, with
$g(x,y)=\min\{|x|,4\}$, give
\[
 u^\eps(0,0,1)-u(0,0,1)
 \geq c\,\frac{\eps^{1/3}}
 {\log^{4/3}(e+\eps^{-1})}.
\]
This proves the result.
\end{proof}

\section{The one-dimensional case}\label{sec:1D}

\subsection{Lipschitz selection of correctors}
We proceed to prove Theorem \ref{thm:1D-Lip}.

We first suppose that $H$ is smooth.  
Given $p\in\R$ and $\la,\nu>0$, let $v_p^{\la,\nu}$ denote the unique periodic solution of
\begin{equation}\label{eq:viscdisc}
 \la v_p^{\la,\nu}
 +H\bigl(y,p+(v_p^{\la,\nu})'\bigr)
 =\nu (v_p^{\la,\nu})''
 \qquad\text{in }\T.
\end{equation}
Set
\begin{equation}\label{eq:total-slope}
 W_p^{\la,\nu}:=p+(v_p^{\la,\nu})'.
\end{equation}
Periodicity gives
\begin{equation}\label{eq:mean-slope}
 \int_\T W_p^{\la,\nu}(y)\,d y=p.
\end{equation}
Differentiating \eqref{eq:viscdisc} and using
$(v_p^{\la,\nu})'=W_p^{\la,\nu}-p$ yields
\begin{equation}\label{eq:W-equation}
 -\nu (W_p^{\la,\nu})''
 +\frac{d}{d y}H\bigl(y,W_p^{\la,\nu}\bigr)
 +\la W_p^{\la,\nu}
 =\la p.
\end{equation}
The important point is that the explicit $y$-dependence remains inside an
exact spatial derivative.

\begin{lemma}\label{lem:nocross}
For every $\la,\nu>0$ and every $p<q$,
\begin{equation}\label{eq:visc-order}
 W_p^{\la,\nu}(y)\le W_q^{\la,\nu}(y)
 \qquad\text{for every }y\in\T.
\end{equation}
\end{lemma}

\begin{proof}
Write $W=W_p^{\la,\nu}$, $Z=W_q^{\la,\nu}$, and $U=W-Z$.
Subtracting the two equations \eqref{eq:W-equation} gives
\begin{equation}\label{eq:U-equation}
 -\nu U''
 +\frac{d}{dy}\bigl(H(y,W)-H(y,Z)\bigr)
 +\la U
 =\la(p-q)<0.
\end{equation}
Moreover, by \eqref{eq:mean-slope},
\[
 \int_\T U\,d y=p-q<0.
\]
Suppose that $U>0$ somewhere.  
Since $U$ cannot be positive on the whole $\T$, the positivity set has a connected component $(a,b)$, interpreted on a lift of $\T$, such that
\[
 U(a)=U(b)=0,
 \qquad
 U>0\quad\text{in }(a,b).
\]
Consequently,
\[
 U'(a)\ge0,
 \qquad
 U'(b)\le0.
\]
At both endpoints $W=Z$, and hence
\[
 H(a,W(a))-H(a,Z(a))=
 H(b,W(b))-H(b,Z(b))=0.
\]
Integrating \eqref{eq:U-equation} on $(a,b)$ therefore gives
\begin{equation}\label{eq:cross-contradiction}
 -\nu\bigl(U'(b)-U'(a)\bigr)
 +\la\int_a^bU(y)\,d y
 =\la(p-q)(b-a).
\end{equation}
The left-hand side is strictly positive, whereas the right-hand side is strictly negative.  
This contradiction proves \eqref{eq:visc-order}.
\end{proof}

For fixed $\la>0$, let $v_p^\la$ be the unique periodic viscosity solution of the first-order discounted equation
\begin{equation}\label{eq:disc}
 \la v_p^\la+H\bigl(y,p+(v_p^\la)'\bigr)=0
 \qquad\text{in }\T.
\end{equation}
Standard vanishing-viscosity stability and comparison imply
\begin{equation}\label{eq:visc-limit}
 v_p^{\la,\nu}\longrightarrow v_p^\la
 \qquad\text{uniformly on }\T
 \quad\text{as }\nu\to0.
\end{equation}
The inequality \eqref{eq:visc-order} can be written in the distributional form
\[
 p+(v_p^{\la,\nu})'
 \le q+(v_q^{\la,\nu})'.
\]
Uniform convergence in \eqref{eq:visc-limit} permits passage to the limit in distributions.  
Since discounted solutions are Lipschitz by coercivity, we obtain, for $p<q$,
\begin{equation}\label{eq:disc-order}
 p+(v_p^\la)'\le q+(v_q^\la)'
 \qquad\text{a.e. on }\T.
\end{equation}

Normalize the discounted solutions by
\begin{equation}\label{eq:disc-normalized}
 \widetilde v^\la(y,p):=v_p^\la(y)-v_p^\la(0).
\end{equation}

\begin{lemma}[Exact Lipschitz estimate at fixed discount]\label{lem:discLip}
For every $\la>0$ and all $p,q\in\R$,
\begin{equation}\label{eq:discLip}
 \norm{\widetilde v^\la(\cdot,p)-\widetilde v^\la(\cdot,q)}_{L^\infty(\T)}
 \le \abs{p-q}.
\end{equation}
\end{lemma}

\begin{proof}
Assume $p<q$ and let $d=q-p>0$.  
Define
\[
 F^\la(y)
 :=q+(v_q^\la)'(y)-p-(v_p^\la)'(y).
\]
By \eqref{eq:disc-order}, $F^\la\ge0$ a.e.  
Periodicity gives 
\begin{equation}\label{eq:Fmass}
 \int_0^1F^\la(y)\,d y=q-p=d.
\end{equation}
Hence, for $0\le y\le1$ fixed,
\[
 0\le\int_0^yF^\la(s)\,d s\le d.
\]
By the fundamental theorem of calculus and the above inequality,
\begin{align}
 \widetilde v^\la(y,q)-\widetilde v^\la(y,p)
 &=\int_0^y\bigl((v_q^\la)'-(v_p^\la)'\bigr)(s)\,d s\notag\\
 &=\int_0^yF^\la(s)\,d s-yd \in [-yd,d-yd].
 \label{eq:difference-formula}
\end{align}
Thus,
\[
-d\le -yd\le
 \widetilde v^\la(y,q)-\widetilde v^\la(y,p)
 \le(1-y)d \le d,
\]
which implies \eqref{eq:discLip}.
\end{proof}

\begin{remark}
The proof also gives the exact identity, for $p<q$,
\[
 \norm{\bigl(q+(v_q^\la)'\bigr)
       -\bigl(p+(v_p^\la)'\bigr)}_{L^1(\T)}=q-p.
\]
\end{remark}

\begin{proof}[Proof of Theorem \ref{thm:1D-Lip}]
The arguments are quite standard (see \cite[Chapter 4]{Tran} for example).

We first retain the smoothness assumption on $H$.  
For $p$ in a compact interval, the normalized discounted correctors are uniformly Lipschitz in $y$, independently of $\la$.  
Indeed, testing \eqref{eq:disc} at a maximum and a minimum of $v_p^\la$ gives
\begin{equation}\label{eq:lambda-v-bound}
 \min_{y\in\T}H(y,p)
 \le -\la v_p^\la(y)
 \le \max_{y\in\T}H(y,p)
 \qquad\text{for every }y\in\T.
\end{equation}
Then, uniform coercivity yields, for every $R>0$, a constant $C_R>0$ such that, uniformly in $\la\in(0,1]$, for $|p|\le R$,
\begin{equation}\label{eq:slope-compact}
 \abs{p+(v_p^\la)'(y)}\le C_R
 \qquad\text{a.e.}
\end{equation}

The estimate \eqref{eq:discLip} and \eqref{eq:slope-compact} make
$\{\widetilde v^\la\}$ equicontinuous on every compact subset of
$\T\times\R$.  By Arzel\`a--Ascoli and a diagonal argument, there is one
sequence $\la_j\to0$ and a continuous function $v(y,p)$ such that
\begin{equation}\label{eq:joint-limit}
 \widetilde v^{\la_j}\longrightarrow v
 \qquad\text{locally uniformly on }\T\times\R.
\end{equation}
Then, for every $p$,
\[
 H\bigl(y,p+v_y(y,p)\bigr)=\Hbar(p)
 \qquad\text{in }\T.
\]
The normalization $v(0,p)=0$ follows from
\eqref{eq:disc-normalized}.  
Passing to the limit in \eqref{eq:discLip} proves \eqref{eq:Lip-selection}.  
Passing to the limit in the distributional inequality \eqref{eq:disc-order} proves \eqref{eq:slope-order}.

For merely continuous $H$, choose smooth periodic Hamiltonians $H_k$ converging locally uniformly to $H$ and sharing a common coercivity modulus.  
Apply the preceding construction to $H_k$.  
On compact $p$-intervals, coercivity supplies uniform spatial Lipschitz bounds, while the parameter estimate has the same constant $1$ for every $k$.  
A second diagonal compactness argument, together with stability of periodic ergodic constants and viscosity solutions, passes the selection and both estimates to $H$.  

To prove the final claim, we note that \eqref{eq:Lip-selection} is precisely \cite[Condition (4.55)]{Tran}.
We apply \cite[the proof of Theorem 4.40]{Tran} to get the desired $O(\eps^{1/2})$ convergence rate.
\end{proof}

\subsection{Optimality of the convergence rate $O(\eps^{1/2})$}

Set $\sigma(y)=2\dist(y,\Z)$ and define
\begin{equation}\label{eq:1D-opt-H}
 H(y,q)=
 \begin{cases}
  -q, &q\leq0,\\
  \min\{2q,\max\{2q-1,\sigma(y)\}\}, &0\leq q\leq1,\\
  q, &q\geq1.
 \end{cases}
\end{equation}
Let $g(x)=\min\{(x)_+,4\}$.

\begin{figure}[htbp]
    \centering
    \includegraphics[width=0.75\textwidth]{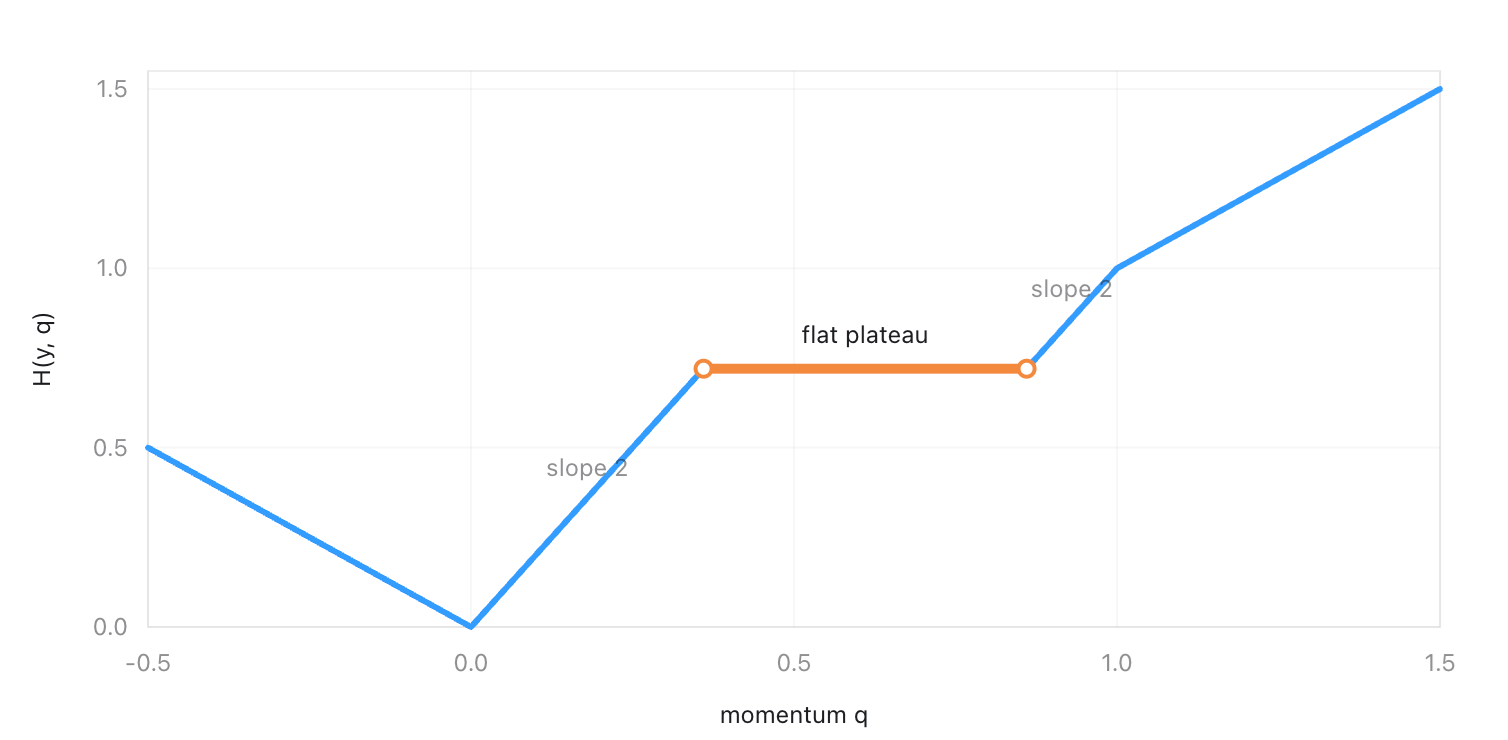}
    \caption{The graph of $H$ (for $y=0.36$).}
    \label{fig:hamiltonian_profile}
\end{figure}

\begin{theorem}\label{thm:1D-optimal}
The Hamiltonian $H\in C^{0,1}(\R\times\R)$ defined in
\eqref{eq:1D-opt-H} is $\Z$-periodic in its first variable and coercive and nonconvex in its second.  
Moreover,
$g\in\BUC(\R)\cap\Lip(\R)$.  
Let $u^\eps$ and $u$ be the corresponding microscopic and effective solutions.  
Then
\begin{equation}\label{eq:1D-opt-limit}
 \liminf_{\eps\to0}
 \frac{|u^\eps(1,1)-u(1,1)|}{\eps^{1/2}}
 \geq\frac{1}{4\sqrt{2\pi}}.
\end{equation}
\end{theorem}
For the proof of Theorem \ref{thm:1D-optimal}, the properties of $H$ follow immediately from \eqref{eq:1D-opt-H}.
In particular, $H(y,q)=|q|$ for $q\notin[0,1]$, and the change of
slope from $2$ to $0$ shows nonconvexity whenever
$0<\sigma(y)<1$.
It is worth noting that $H$ is level-set quasiconvex in $q$.
The key geometric fact here is that, for $p\in (0,1)$,
\begin{equation}\label{eq:key geometric fact}
H\left(y,\frac p2+\frac12\mathbf 1_{\{\sigma(y)<p\}} \right) = p.
\end{equation}
\begin{lemma}\label{lem:1D-Hbar}
For the Hamiltonian defined in \eqref{eq:1D-opt-H},
\[
    \Hbar(p)=p \qquad \text{for all }p\in[0,1],
\]
and the effective solution with initial datum $(x)_+$ is
\begin{equation}\label{eq:1D-effective-ramp}
 u_{\rm ramp}(x,t)=(x-t)_+.
\end{equation}
\end{lemma}

\begin{proof}
The endpoint identities $\Hbar(0)=0$ and $\Hbar(1)=1$ follow by taking $v\equiv0$, since $H(y,0)=0$ and $H(y,1)=1$ for every $y$.

For $p\in(0,1)$, set
\begin{equation}\label{eq:1D-physical-slope}
 q(y,p)=\frac p2+\frac12\mathbf 1_{\{\sigma(y)<p\}},
 \qquad
 v(y,p)=\int_0^y(q(r,p)-p)\,d r.
\end{equation}
Since $|\{y\in\T:\sigma(y)<p\}|=p$, the function $v(\cdot,p)$ is periodic.
Away from the two crossing points,
\[
 H\bigl(y,p+v_y(y,p)\bigr)=H(y,q(y,p))=p.
\]

As $H$ is level-set quasiconvex in $q$, thanks to \cite[Chapter 2]{Tran}, we only need to check the supersolution property at $y=z$, the corner from below (convex corner) of the graph of $v(\cdot,p)$.
It is clear that $p+D^-v(z,p)=[p/2,(p+1)/2]$ and $\sigma(z)=p$.
Since $H(y,\cdot)=p$ on this interval, both viscosity inequalities hold.
Thus, for $p\in(0,1)$,
\begin{equation}\label{eq:1D-Hbar}
 \Hbar(p)=p.
\end{equation}
\end{proof}
\begin{proof}[Proof of Theorem \ref{thm:1D-optimal}]
Let $U$ solve the fast equation
\begin{equation}\label{eq:1D-fast-ramp}
 \begin{cases}
  U_s+H(y,U_y)=0 &\text{in }\R\times(0,\infty),\\
  U(y,0)=(y)_+ &\text{on }\R.
 \end{cases}
\end{equation}
The proof strategy for the lower bound is to construct a subsolution of \eqref{eq:1D-fast-ramp}.

Put $\gamma=1/4$.  
For any convex function $a:\R\to\R$ satisfying
\[
 0\leq a'\leq1,\qquad
 \lim_{z\to-\infty}a(z)=0,\qquad
 \lim_{z\to\infty}(a(z)-z)=0,
\]
define
\begin{equation}\label{eq:barrier-T}
 (\mathcal Ta)(z)=\int_0^\infty e^{-\rho}
 a\bigl(z+\gamma(\rho-1)\bigr)\,d\rho.
\end{equation}
\par\vspace{\baselineskip}
\noindent\emph{Sketch of the proof.}
For any such $a$, we will first construct $A_a:[0,1]\times\R\to\R$ with
$A_a(0,z)=a(z)$ and $A_a(1,z)=(\mathcal Ta)(z)$.  
We then show that
$V_a(y,s):=A_a(y,y-s)$ is a subsolution for $0<y<1$.  We start with $a_0(z)=z_+$, iterate according to $a_{n+1}=\mathcal Ta_n$ and let  $V(y,s)=A_{a_n}(y-n,y-s)$. After extending $V$ by zero for $y\leq0$, the resulting function $V$ is a global subsolution. Finally, we
estimate $V(n,n)=a_n(0)$ and rescale this estimate to obtain the
claimed lower bound.

\par\vspace{\baselineskip}
\noindent\emph{Construction of $A_a$ and $V_a$.}
By approximations and stability of viscosity solutions, we can assume that $F=a'$ is continuous and strictly increasing on $\{0<F<1\}$.  
For $0<\tau<1$, define $d(\tau)$ by
$F(d(\tau))=\tau$.  
Next, set
\[
 G(x)=\int_0^\infty e^{-\rho}F(x+\gamma\rho)\,d\rho=\frac{e^{x / \gamma}}{\gamma} \int_x^{\infty} e^{-r / \gamma} F(r) d r.
\]
Then, $F\leq G$, $G$ is strictly increasing on $\{0<G<1\}$, and one has
\begin{equation}\label{eq:G-ode}
 \begin{cases}
G-\gamma G'=F,\\
G(-\infty)=0,\ G(+\infty)=1.
 \end{cases}
\end{equation}
Thus there is a unique $b(\tau)$ such that
$G(b(\tau))=1-\tau$.

We now define two functions.  
For $0<\tau<1$, let
\begin{equation}\label{eq:barrier-Cplus}
 C_+(\tau,x)=\min_{r\in [0,\gamma]}(a(x+r)+\tau(\gamma-r))=
 \begin{cases}
  a(x+\gamma),\qquad&x\leq d(\tau)-\gamma,\\
  a(d(\tau))+\tau(x-d(\tau)+\gamma),
       \qquad&d(\tau)-\gamma\leq x\leq d(\tau),\\
  a(x)+\gamma\tau,\qquad&x\geq d(\tau).
 \end{cases}
\end{equation}
Also, let
\begin{equation}\label{eq:barrier-Cminus}
 C_-(\tau,x)=
 a(x)+\gamma(G(x)-(1-\tau))_+=
 \begin{cases}
  a(x),\qquad&x\leq b(\tau),\\
a(b(\tau))+\displaystyle\int_{b(\tau)}^xG(\xi)\,d\xi, \qquad&x\geq b(\tau).
 \end{cases}
\end{equation}
The endpoint values of $C_+$ and $C_-$ are defined by continuity.
Using these functions, define
\begin{equation}\label{eq:barrier-A}
 A_a(y,z)=
 \begin{cases}
  C_+(2y,z-2\gamma y),\qquad&0\leq y\leq1/2,\\
  C_-(2y-1,z-\gamma(2y-1)),\qquad&1/2\leq y\leq1,
 \end{cases}
\end{equation}
and set 
\[
V_a(y,s)=A_a(y,y-s)=
 \begin{cases}
  C_+(2y,\frac{y}{2}-s),\qquad&0\leq y\leq1/2,\\
  C_-(2y-1,\frac{y}{2}+\frac{1}{4}-s),\qquad&1/2\leq y\leq1,
 \end{cases}
 \]
 which utilizes the geometric fact \eqref{eq:key geometric fact}.

\par\vspace{\baselineskip}
\noindent\emph{Verification that $V_a$ is a subsolution.}
In either line of
\eqref{eq:barrier-A}, write $C=C_+$ or $C=C_-$, and set
$P=C_x$.  The chain rule gives
\begin{equation}\label{eq:barrier-gradients}
 (V_a)_s=-P,\qquad (V_a)_y=\frac P2+2C_\tau.
\end{equation}
If $0<y<1/2$, then $\tau=\sigma(y)=2y$.
The three regions in \eqref{eq:barrier-Cplus} give, respectively,
\[
 (C_\tau,P)=(0,P),\qquad
 (C_\tau,P)=(x-d(\tau)+\gamma,\tau),\qquad
 (C_\tau,P)=(\gamma,P),
\]
where $P\leq\tau$ in the first region, $0\leq x-d(\tau)+\gamma\leq\gamma$ in the middle, and $P\geq\tau$ in the third.
Thus, $(V_a)_y$ is $P/2$, belongs to
$[\tau/2,(\tau+1)/2]$, or is $(P+1)/2$, respectively.
In all three cases,
\[
 H(y,(V_a)_y)=P=-(V_a)_s.
\]

If $1/2<y<1$, then $\tau=2y-1$ and
$\sigma(y)=1-\tau$.  Differentiating
$G(b(\tau))=1-\tau$ and using \eqref{eq:G-ode}, we obtain
\[
 b'(\tau)=-\frac{\gamma}{1-\tau-F(b(\tau))}.
\]
Here, the denominator is $\gamma G'(b(\tau))>0$.
It follows from \eqref{eq:barrier-Cminus} that
\begin{align*}
 C_\tau=0,\quad P=F(x)\leq1-\tau
 \qquad&\hbox{if }x<b(\tau),\\
 C_\tau=\gamma,\quad P=G(x)\geq1-\tau
 \qquad&\hbox{if }x>b(\tau).
\end{align*}
Equation \eqref{eq:barrier-gradients} again gives
$H(y,(V_a)_y)=P=-(V_a)_s$.  
At $x=b(\tau)$, the left and right
slopes are $F(b(\tau))\leq1-\tau$ and
$G(b(\tau))=1-\tau$.  
Hence, $C_-$ has a corner from below (a convex corner).
If a $C^1$ function touched $V_a$ from above at this point, its restriction to fixed $y$ would touch this convex corner from above, which is impossible.  
Thus, the subsolution inequality holds there as well.

\par\vspace{\baselineskip}
\noindent\emph{The traces of $A_a$.}
Taking limits in
\eqref{eq:barrier-Cplus}--\eqref{eq:barrier-A} gives
\[
 A_a(0,z)=A_a(1/2,z)=a(z),\qquad
 A_a(1,z)=\int_{-\infty}^{z-\gamma}G(\xi)\,d\xi=(\mathcal Ta)(z).
\]
For the last equality, both functions vanish at $-\infty$, and their derivatives are equal to $G(z-\gamma)$.
At $y=1/2$, the one-sided spacetime gradients agree when $0<P<1$.  
If $P=1$, all possible spatial slopes belong to
$[1/2,1]$, where $H(1/2,\cdot)=1$.  
Thus, the subsolution inequality also holds at $y=1/2$.  
The construction gives $-1\leq (V_a)_s\leq0$ and $0\leq (V_a)_y\leq1$ a.e..


\par\vspace{\baselineskip}
\noindent\emph{Iteration of the construction.} Set
\[
 a_0(z)=z_+,\qquad a_{n+1}=\mathcal Ta_n,
\]
and let $A_n=A_{a_n}$.  
We add the subscript $n$ to $F,G,b,C_+$, and $C_-$ when they are constructed from $a_n$.
For $n\geq 0$ and $n\leq y\leq n+1$, define
\begin{equation}\label{eq:barrier-global-V}
 V(y,s)=A_n(y-n,y-s).
\end{equation}
At each interface $y=n$, the one-sided spacetime gradients agree for $0<P<1$; the remaining case has $P=0$ and spatial slopes in $[0,1/2]$, where $H(n,\cdot)=0$.  
The endpoint identities and periodicity of $H$ therefore show that $V$ is a subsolution on $y>0$.

The operator $\mathcal T$ preserves convexity, the slope bounds, and
the two limits imposed on $a$.  
It also follows by induction that \(a_n(z)=z\) for \(z\geq n/4\).  
Indeed, if $z\geq(n+1)/4$, every argument in \eqref{eq:barrier-T} is at least $n/4$, and integration gives $a_{n+1}(z)=z$.  

For $n\leq y\leq n+1/2$, put $\tau=2(y-n)$; for $n+1/2\leq y\leq n+1$, put
$\tau=2(y-n)-1$.  
At $s=0$, the corresponding $x$-coordinates are $n+\gamma\tau$ and $n+1/2+\gamma\tau$.  
Both lie in the linear region of $a_n$, so $F_n(x)=G_n(x)=1$.  
Hence, the last region of \eqref{eq:barrier-Cplus} gives $C_{n,+}(\tau,x)=x+\gamma\tau$.  
In the second line, integrating \eqref{eq:G-ode} from $b_n(\tau)$ to $x$ gives
$C_{n,-}(\tau,x)=a_n(x)+\gamma(G_n(x)-G_n(b_n(\tau)))=x+\gamma\tau$.
Since $x+\gamma\tau=y$, we have $V(y,0)=y$ for $y\geq0$.

Set $V(y,s)=0$ for $y\leq0$.  
For each $s>0$ and all sufficiently small $y>0$, the first formula gives
\[
 V(y,s)=y(y+1/2-2s)_+.
\]
If $0<s<1/4$, this function has a convex corner at $y=0$, and hence there is no test function touching it from above.  
If $s>1/4$, it is locally zero.  
At $s=1/4$, its right-hand profile is $y^2$, and every test function touching from above has spacetime gradient $(0,0)$. 
On $y<0$, the function is identically zero and $H(y,0)=0$.  
Thus, $V$ is a global viscosity subsolution of \eqref{eq:1D-fast-ramp} with initial datum $(y)_+$.  
Moreover, $-1\leq V_s\leq0$ and the initial condition imply $|V(y,s)-(y)_+|\leq s$.  
The comparison principle therefore gives $V\leq U$.

\par\vspace{\baselineskip}
\noindent\emph{Estimate of $U(T,T)$.}
Since the $n$-fold convolution of $e^{-\rho}\mathbf1_{\{\rho>0\}}$ is
$e^{-\rho}\rho^{n-1}/(n-1)!$, for $n\geq1$,
\begin{equation}\label{eq:barrier-gamma}
 a_n(z)=\frac1{(n-1)!}\int_0^\infty
 e^{-\rho}\rho^{n-1}\left(z+\frac{\rho-n}{4}\right)_+\,d\rho.
\end{equation}
Therefore
\begin{align}
 U(n,n)&\geq V(n,n)=a_n(0)\notag\\
 &=\frac1{4(n-1)!}\int_n^\infty
       (\rho-n)\rho^{n-1}e^{-\rho}\,d\rho\notag\\
 &=\frac{n^ne^{-n}}{4(n-1)!}
   \sim\frac{\sqrt n}{4\sqrt{2\pi}}.                 \label{eq:1D-contact-value}
\end{align}
Here, the last identity follows by differentiating $\rho^ne^{-\rho}$, and the asymptotic estimate follows from Stirling's formula.  
For $T>1$, let $N=\lfloor T\rfloor$.
By the Lipschitz bounds above,
\[
 U(T,T)\geq V(T,T)\geq V(N,N)-2=a_N(0)-2\sim \frac{\sqrt{N}}{4 \sqrt{2 \pi}}\ge \frac{\sqrt{T}}{4 \sqrt{2 \pi}}.
\]

\par\vspace{\baselineskip}
\noindent\emph{Conclusion.}
Let $u_{\rm ramp}^\eps$ denote the microscopic solution with initial
datum $(x)_+$.  
The exact scaling $u_{\rm ramp}^\eps(x,t)=\eps U(x/\eps,t/\eps)$, with
$T=\eps^{-1}$, now gives
\[
 \liminf_{\eps\to0}
 \frac{u_{\rm ramp}^\eps(1,1)-u_{\rm ramp}(1,1)}{\eps^{1/2}}
 \geq\liminf _{T \rightarrow \infty} \frac{U(T, T)}{\sqrt{T}} 
 \geq\frac1{4\sqrt{2\pi}}.
\]

Finally, both $H(y,\cdot)$ and $\Hbar$ are $2$-Lipschitz.  
Thus, the finite propagation at $(1,1)$ uses only $[-1,3]$, where $g=(x)_+$.
Hence, $u^\eps(1,1)=u_{\rm ramp}^\eps(1,1)$ and
$u(1,1)=u_{\rm ramp}(1,1)=0$.  
We therefore obtain \eqref{eq:1D-opt-limit}.
\end{proof}

\appendix
\section{No continuous selection of correctors in two dimensions} \label{appendix}
We give an example in two dimensions to show that there is no continuous selection of correctors in two dimensions.
Previously, \cite{MTY} gave one such example in three dimensions.
Because of this, it is not possible to use the approach of \cite{CDI} directly to obtain the convergence rate $o(\eps^{1/3})$ posed in Open Problem \ref{Op2}.

Let \(X=(x,y)\in\T^2\) and \(P=(p,q)\in\R^2\). 
Define
\[
H(X,P)
=
p^2+q^2
+
p\cos(2\pi y)
-
\sin^2(2\pi y).
\]
For \(a\in\R\), let $P_a=(a,0)$.
We will prove that the corrector associated with \(P_a\) is unique up to
an additive constant whenever \(a\neq0\).

\begin{lemma}
For fixed $a\in \R$,
\[
    \overline H(P_a)=a^2+|a|.
    \]
\end{lemma}

\begin{proof}
Define 
\[
F_a(y)
=
\sin^2(2\pi y)
+
|a|-a\cos(2\pi y).
\]
Then it is easy to see that $\min_{y\in \T} F_a(y)=0$ and the minimum is attained only at $y=0$ if $a>0$ and at $y=1/2$ if $a<0$. 
Hence, there exists a unique Lipscthiz continuous viscosity solution $\phi_a=\phi_a(y)\in W^{1,\infty}(\T)$ such that
\[
|\phi_a'|^2-F_a(y)=0 \qquad \text{in $\T$}.
\]
See \cite[Chapter 4]{Tran} for example.
Accordingly, $\tilde \phi_a(x,y)=\phi_a(y)$ is also a viscosity solution of 
\[
H(x,y,P_a+D\tilde \phi_a)=a^2+|a| \qquad \text{in $\T^2$}.
\]    
\end{proof}

\begin{lemma}\label{lem:aubry}
Let $\mathcal{A}_{P_a}$ be the Aubry set at $P_a$.
Then,
\begin{align}
    \mathcal{A}_{P_a}=\T\times \{0\} \qquad &\text{if $a>0$},\\[3mm]
    \mathcal{A}_{P_a}=\T\times \left\{{\frac{1}{2}}\right\} \qquad &\text{if $a<0$}.
\end{align}   
Hence, the cell problem associated with $P_a$ has a unique solution up to a constant for $a\neq0$. 
\end{lemma}

\begin{proof}
It suffices to consider the case $a>0$.
Let $w={\tilde \phi_a/2}$. 
Then $w$ is a viscosity subsolution of 
\[
H(x, y,P_a+D w)= a^2+|a| \qquad \text{in $\T^2$}.
\]
and the equality only holds when $(x,y)\in \T\times \{0\}$. 
Therefore,
\[
\mathcal{A}_{P_a}\subset \T\times \{0\}. 
\]
On the other hand, $D w=(0,0)$ on $\T\times \{0\}$, which together with the flow invariance property of the Aubry set, implies the conclusion.
See \cite{Fathi, Tran, TYWKAM} for more details on the Aubry set.
\end{proof}

\begin{proposition}
Let \(X=(x,y)\in\T^2\) and \(P=(p,q)\in\R^2\). 
Define
\[
H(X,P)
=
p^2+q^2
+
p\cos(2\pi y)
-
\sin^2(2\pi y).
\]
Then, there is no continuous selection of correctors to the cell problems.     
\end{proposition}

\begin{proof}
For $a\not=0$, in light of Lemma \ref{lem:aubry}, the solution to the cell problem associated with $P_a$ that is zero at the origin must be given by 
\[
v_a(x,y)=\phi_a(y)-\phi_a(0).
\]
As $a\to 0$, the  two limits are 
\[
\mathrm{sign}(a)\frac{1-\cos (2\pi y)}{2\pi}.
\]   
Thus, there is no continuous selection of the correctors associated with $P_a=(a,0)$ as $a\to 0$.
\end{proof}

\section*{Acknowledgment}
We would like to thank Professor Yeoneung Kim for providing various prompts and independent checks of some of our results and mechanisms via ChatGPT 5.6 Sol.

\end{document}